\documentclass[11pt, a4paper]{amsart}

\usepackage{amssymb}
\usepackage{color}
\usepackage{amsmath,amssymb,hyperref}
\usepackage{amsfonts}
\usepackage{enumerate}
\usepackage{amsmath, amsthm}
\usepackage{amscd}
\usepackage{graphicx}
\hypersetup{pdfpagemode=FullScreen,colorlinks=true}

\input xy
\xyoption{all}

      \theoremstyle{plain}
      \newtheorem{theorem}{Theorem}[section]
      \newtheorem{lemma}[theorem]{Lemma}
      \newtheorem{corollary}[theorem]{Corollary}
      \newtheorem{proposition}[theorem]{Proposition}
	
      \theoremstyle{definition}
      \newtheorem{definition}[theorem]{Definition}
      \newtheorem*{remark}{Remark}

\def\co{\colon\thinspace}
\def\SO{\mathrm{SO}(n,1)}
\def\H{\mathbb H^n}

\def\wM{\widehat{M}}
\def\cH{\overline{\mathbb{H}}^n}

\def\vol{\mathrm{Vol}}

\def\T{\mathcal{T}}
\def\Tub{\mathrm{Tub}}
\def\PSL{\mathrm{SO}(3,1)}
\def\O{\mathcal{O}}
\def\oM{\overline{M}}
\def\hom{\mathrm{Hom}}

      \makeatletter
      \def\@setcopyright{}
      \def\serieslogo@{}
      \makeatother

\begin{document}

%


   \author{Sungwoon Kim}
   \address{School of Mathematics,
   KIAS, Hoegiro 85, Dongdaemun-gu,
   Seoul, 130-722, Republic of Korea}
   \email{sungwoon@kias.re.kr}

   \author{Inkang Kim}
   \address{School of Mathematics,
   KIAS, Hoegiro 85, Dongdaemun-gu,
   Seoul, 130-722, Republic of Korea}
   \email{inkang@kias.re.kr}





   \title[On deformation spaces of nonuniform hyperbolic lattices]{On deformation spaces of nonuniform hyperbolic lattices}


\begin{abstract}
Let $\Gamma$ be a nonuniform lattice acting on real hyperbolic $n$-space.
We show that in dimension greater than or equal to $4$, the volume of a representation is constant on each connected component of the representation variety of $\Gamma$ in $\SO$.
Furthermore, in dimensions $2$ and $3$, there is a semialgebraic subset of the representation variety such that the volume of a representation is constant on connected components of the semialgebraic subset.
Our approach gives a new proof of the local rigidity theorem for nonuniform hyperbolic lattices and the analogue of Soma's theorem, which shows that the number of orientable hyperbolic manifolds dominated by a closed, connected, orientable $3$-manifold is finite, for noncompact $3$-manifolds.
\end{abstract}


\footnotetext[1]{The first author was supported by Basic Science Research Program through the National Research Foundation of Korea(NRF) funded by the Ministry of Education, Science and Technology(NRF-2012R1A1A2040663) and the second author gratefully acknowledges the partial support
of NRF grant  (2010-0024171) and a warm support of IHES during his
stay.
}



   \thanks{}
   \thanks{}

   \dedicatory{}

   \date{}


   \maketitle



\section{Introduction}

Let $\Gamma$ be a lattice in $\SO$. Then an invariant $\vol(\rho)$ is associated to an arbitrary representation $\rho :\Gamma \rightarrow \SO$. This invariant is called the \emph{volume of a representation}. The definition of the volume of a representation depends on whether $\Gamma$ is a uniform lattice or not. In the uniform lattice case, there is a natural way to define the volume of a representation, as follows. To a representation $\rho :\Gamma \rightarrow \SO$, denote by $E_\rho$ the corresponding flat $\H$-bundle over $M=\Gamma\backslash \H$. Let $\omega_{\H}$ be the Riemannian volume form on $\H$ and $\omega_{E_\rho}$ be a closed $n$-form on $E_\rho$ by spreading $\omega_{\H}$ over the fibres of $E_\rho$. Let $s$ be a section of $E_\rho$. Then the volume of $\rho$ is defined as $$\vol(\rho)=\langle s^*\omega_{E_\rho}, [M] \rangle ,$$
where $s^*\omega_{E_\rho}$ is considered as a singular cohomology class in $H^n(M)$ and $[M]$ denotes the fundamental class of $M$. This definition is independent of the choice of a section.

Unfortunately, the definition as above does not work in the nonuniform lattice case. It has been modified to define the volume of a representation for nonuniform lattices in several ways. Dunfield \cite{Du99} first introduced the notion of pseudo-developing map to define the volume of a representation of a nonuniform lattice in $\mathrm{SO}(3,1)$. Then Francaviglia \cite{Fra}
proved that the definition of the volume of a representation given by Dunfield is well-defined for any representation of a nonuniform lattice in $\mathrm{SO}(3,1)$. The authors \cite{SKIK} give a definition of the volume of a representation for a nonuniform lattice in a semisimple Lie group through bounded cohomology, $\ell^1$-homology and simplicial volume. Bucher--Burger--Iozzi \cite{BBI} give a definition of the volume of a representation of a nonuniform hyperbolic lattice in $\SO$ in the language of bounded cohomology.
In fact, in the case of hyperbolic lattices, all definitions of the volume of a representation give the same value. For further details, see \cite[Section 6]{SKIK} (a similar proof works for any dimension) and Section \ref{sec:volume}.

As one can see, the definition of the volume of a representation is different depending on the uniform lattice or nonuniform lattice. However in either case, the volume satisfies a Milnor--Wood type inequality $$\vol(\rho) \leq \vol(M).$$
Moreover, the equality holds if and only if $\rho$ is conjugate to the lattice embedding $i:\Gamma \hookrightarrow \SO$ by an isometry, provided $n\geq 3$. This rigidity result for maximal representations recovers Mostow rigidity for hyperbolic manifolds. Dunfield \cite{Du99} provided a proof of the volume rigidity theorem in dimension $3$, following the proof of Thurston's strict version \cite[Theorem 6.4]{Th78} of Mostow's theorem with some details coming from Toledo's paper \cite{Tol}. Then Francaviglia and Klaff \cite{FK} proved a volume rigidity theorem for representations $\rho : \Gamma \rightarrow \mathrm{SO}(m,1)$ for $m\geq n\geq 3$.
Bucher--Burger--Iozzi \cite{BBI} give a complete proof of volume rigidity from the viewpoint of bounded cohomology.

The study of the set of values for the volume of a representation is closely related to the local rigidity theorem for hyperbolic manifolds. In the uniform lattice case, it is well known that the volume of a representation is constant on each connected component of $\mathrm{Hom}(\Gamma,\SO)$.
Hence the set is discrete. This actually follows from the rigidity of secondary characteristic classes associated to a flat connection on a hyperbolic manifold. See \cite{Rez} by Reznikov for details. Besson--Courtois--Gallot \cite{BCG} gave a geometric proof for this by using the Schl\"{a}fli formula. The constancy of the volume of a representation on connected components and the volume rigidity theorem recover the local rigidity theorem for uniform hyperbolic lattices.

The set of values for the volume of a representation in the nonuniform lattice case is a little different. The set of values for the volume of a representation in dimensions $2$ or $3$ is not discrete anymore. In fact, an open interval is contained in the set of values in those cases. In even dimensions greater than or equal to $4$, Bucher--Burger--Iozzi prove that the set of values for the volume of a representation is an integer up to a universal constant. In accordance with their result, it is natural to expect that the set of values for the volume of a representation is discrete in any dimension greater than or equal to $4$. However, this has not yet been known. The aim of the paper is to explore the values of the volume of a representation on the representation variety of a nonuniform hyperbolic lattice in $\SO$.

\begin{theorem}\label{thm:1.1}
Let $n\geq 4$ and $\Gamma$ be a nonuniform lattice in $\SO$. Let $\rho_t : \Gamma \rightarrow \SO$ be a $C^1$-smooth path on $\hom(\Gamma,\SO)$. Then $\vol(\rho_t)$ is constant.
\end{theorem}

For $n\geq 4$, it is well known that a nonuniform lattice $\Gamma$ is locally rigid in $\SO$. Note that this local rigidity theorem does not follow from Mostow's rigidity theorem. There are two methods to prove the local rigidity theorem. The first one is to compute the first group cohomology $H^1(\Gamma, Ad \circ i)$. If $H^1(\Gamma, Ad \circ i)$ vanishes, then a lattice embedding $i:\Gamma \hookrightarrow \SO$ is locally rigid.
This method was used by Calabi \cite{Cal} and generalized by Weil \cite{Wei60,Wei62}. The second one is to follow Mostow's proof. This was realized by Kapovich \cite{Ka}.
Here we present a third method. The volume rigidity theorem and Theorem \ref{thm:1.1} give a new proof for the local rigidity theorem of nonuniform hyperbolic lattices in dimension greater than or equal to $4$. Actually we obtain the global structure of $\mathrm{Hom}(\Gamma,\SO)$ around a lattice embedding $i :\Gamma \rightarrow \SO$ as follows.

\begin{corollary}\label{cor:1.2}
Let $n\geq 4$ and $\Gamma$ be a nonuniform lattice in $\SO$. Then the connected component of $\mathrm{Hom}(\Gamma,\SO)$ containing a lattice embedding $i :\Gamma \hookrightarrow \SO$ consists of representations conjugate to $i$ by isometries.
\end{corollary}

As mentioned before, it is well known that for a nonuniform lattice $\Gamma$ in $\mathrm{SO}(3,1)$, the volume of a representation is not constant on connected components of $\mathrm{Hom}(\Gamma,\mathrm{SO}(3,1))$ and the local rigidity theorem also fails. However there is no local deformation preserving parabolicity for a lattice embedding $i :\Gamma \rightarrow \mathrm{SO}(3,1)$. For this reason we explore the set of values for the volume of a representation on the subvariety $\mathrm{Hom}_{par}(\Gamma, \mathrm{SO}(3,1))$ consisting of representations of $\Gamma$ in $\mathrm{SO}(3,1)$ preserving parabolicity.

\begin{theorem}\label{thm:1.3}
Let $\Gamma$ be a nonuniform lattice in $\mathrm{SO}(3,1)$ and $\rho_t : \Gamma \rightarrow \mathrm{SO}(3,1)$ be a $C^1$-smooth path on $\mathrm{Hom}_{par}(\Gamma, \mathrm{SO}(3,1))$. Then $\vol(\rho_t)$ is constant.
\end{theorem}

Since $\mathrm{Hom}_{par}(\Gamma, \mathrm{SO}(3,1))$ is an algebraic variety, there are finitely many connected components. Hence the volume of a representation takes a finite number of values on $\mathrm{Hom}_{par}(\Gamma, \mathrm{SO}(3,1))$. The volume rigidity theorem in dimension $3$ and Theorem \ref{thm:1.3} immediately give the following corollary.

\begin{corollary}\label{cor:1.4}
Let $\Gamma$ be a nonuniform lattice in $\mathrm{SO}(3,1)$. Then the connected component of $\mathrm{Hom}_{par}(\Gamma,\mathrm{SO}(3,1))$ containing a lattice embedding $i :\Gamma \hookrightarrow \mathrm{SO}(3,1)$ consists of representations conjugate to $i$ by isometries.
\end{corollary}

Let $M$ and $N$ be connected, orientable manifolds of the same dimension. Then $M$ is said to be \emph{dominated} by $N$ if there exists a proper map $f:N \rightarrow M$ with nonzero degree.
For any compact, connected, orientable $3$-manifold $N$, Soma \cite{So} proved that there are only a finite number of hyperbolic manifolds dominated by $N$. As an application of Theorem \ref{thm:1.3}, we obtain an analogous theorem to Soma's for noncompact $3$-manifolds.

\begin{theorem}\label{thm:1.5}
Let $N$ be a connected, orientable, cusped $3$-manifold. Then there are only a finite number of hyperbolic $3$-manifolds dominated by $N$.
\end{theorem}

Our approach in the paper is basically based on the proof of Besson--Courtois--Gallot \cite{BCG} in the uniform lattice case.
A technical difficulty in dealing with the nonuniform lattice case is to control the behavior of the parabolic subgroups of a nonuniform lattice $\Gamma$ along a given smooth path on $\mathrm{Hom}(\Gamma,\SO)$. To overcome this issue, we define a boundary-compact path and boundary-parabolic path depending on the behavior of a parabolic subgroup of $\Gamma$ along the path. Then we show that the volume of a representation is constant on either a boundary-compact or boundary-parabolic path and generalize this to arbitrary smooth paths on $\mathrm{Hom}(\Gamma,\SO)$ in dimension greater than or equal to $4$. See section \ref{sec:general} for definitions.

\subsection*{Structure of the paper} We recall some basics of the disjoint cone of a compact manifold with boundary, continuous group cohomology, and the Schl\"{a}fli formula in Section \ref{sec:pre}. We reformulate the volume of a representation in the nonuniform lattice case in order to compute the volume of a representation with the triangulation given in Section \ref{sec:triangulation} and the equivariant map constructed in Section \ref{sec:equimap}. This reformulation makes it possible to apply the Schl\"{a}fli formula to our situation. In Section \ref{sec:onecusp}, we prove Theorem \ref{thm:1.1} in the
one cusped hyperbolic manifold case and then extend it to the general case in Section \ref{sec:general}.
Theorem \ref{thm:1.3} and \ref{thm:1.5} are covered in Section \ref{sec:3dimension}. In addition, we deal with the dimension $2$ case in Section \ref{sec:2dimension} in our setting.

\section{Preliminaries}\label{sec:pre}

In this section, we collect some definitions and facts which are used throughout the paper.

\subsection{The disjoint cone of a compact manifold with boundary}
Let $\oM$ be a compact, connected, smooth, oriented $n$-manifold with boundary $\partial \oM$ and let $M$ denote the interior of $\oM$.
Let $\partial_1 \oM,\ldots,\partial_r \oM$ be the connected components of $\partial \oM$.
Then one defines the disjoint cone $$\wM=\mathrm{Dcone}(\cup_{i=1}^r \partial_i \oM \rightarrow \oM)$$ by coning each $\partial_i \oM$ to a point. Note that $\widehat{M}$ can be obtained by collapsing each boundary component of $\oM$ to a point. Then it can be easily seen that there is an isomorphism $$H_*(\oM,\partial \oM) \cong H_*(\wM)$$ in degree $*\geq 2$. In particular $H_n(\wM)\cong \mathbb R$ if $n\geq 2$. Denote a fundamental class of $\wM$ by $[\wM]$. Each boundary of $\oM$ defines a subgroup of $\pi_1(M)$ which is well-defined up to conjugacy. These are called the \emph{peripheral subgroups} of $\oM$.

Let $p : \widetilde{\oM} \rightarrow \oM$ be the universal covering map of $\oM$. Let $\widehat{\widetilde M}$ denote the quotient space of $\widetilde{\oM}$ obtained by collapsing each component of $p^{-1}(\partial \oM)$ to a point. The covering map $p : \widetilde{\oM} \rightarrow \oM$ extends to a map $\hat{p}: \widehat{\widetilde M} \rightarrow \widehat{M}$ and moreover, the action of $\pi_1(M)$ on $\widetilde{\oM}$ by covering transformations induces an action on $\widehat{\widetilde M}$. However, the action on $\widehat{\widetilde M}$ is not free because each point of $\partial  \widehat{\widetilde M}$ is stabilized by some peripheral subgroup of $\pi_1(M)$. Clearly $\pi_1(M) \backslash \widehat{\widetilde M}=\widehat{M}$.

The set of points of $\widehat{\widetilde M}$ coming from the collapsed components of $p^{-1}(\partial \oM)$ is denoted by $\partial \widehat{\widetilde M}$. A point of $\partial \widehat{\widetilde M}$ is called an \emph{ideal point} of $\widehat{\widetilde M}$ and its image in $\wM$ by $\hat p$ is also called an \emph{ideal point} of $\wM$.

\subsection{Continuous group cohomology}

Let $G$ be a topological group. Consider the continuous cocomplex $C^*_c(G,\mathbb{R})$ with homogeneous coboundary operator, where $$C^k_c(G,\mathbb{R})=\{ f \co G^{k+1} \rightarrow \mathbb{R}\text{ }|\text{ }f \text{ is continuous} \}.$$
The action of $G$ on $C^k_c(G,\mathbb{R})$ is given by
$$(g\cdot f)(g_0,\ldots,g_k)=f(g^{-1}g_0,\ldots,g^{-1}g_k).$$
The continuous cohomology $H^*_c(G,\mathbb{R})$ of $G$ with trivial coefficients is defined as the cohomology of the $G$-invariant continuous cocomplex $C^*_c(G,\mathbb{R})^G$.

For a cochain $f\co G^{k+1} \rightarrow \mathbb{R}$, define its sup norm by
$$\|f\|_\infty = \sup \{ |f(g_0,\ldots,g_k)|\text{ }|\text{ } (g_0,\ldots,g_k)\in G^{k+1}\}.$$
The sup norm turns $C^*_c(G,\mathbb{R})$ into a complex of normed real vector spaces. The continuous bounded cohomology $H^*_{c,b}(G,\mathbb{R})$ of $G$ is defined as the cohomology of the subcocomplex $C^*_{c,b}(G,\mathbb{R})^G$ of
$G$-invariant continuous bounded cochains in $C^*_c(G,\mathbb{R})$.
The inclusion of $C^*_{c,b}(G,\mathbb{R})^G \subset C^*_c(G,\mathbb{R})^G$ induces a comparison map $c \co H^*_{c,b}(G,\mathbb{R}) \rightarrow H^*_c(G,\mathbb{R})$. The sup norm induces seminorms on both $H^*_c(G,\mathbb{R})$ and $H^*_{c,b}(G,\mathbb{R})$, denoted by $\| \cdot \|_\infty$.

Now let $G$ be a semisimple Lie group and $X$ the associated symmetric space. We here describe other useful cocomplexes for both the continuous and the continuous bounded cohomology of $G$. For a nonnegative integer $k$, define
$$C^k_c(X,\mathbb{R}) = \{ f \co X^{k+1} \rightarrow \mathbb{R} \ | \ f \text{ is continuous} \}.$$
The $G$-action on $C^*_c(X,\mathbb{R})$ is defined analogously to the one on $C^*_c(G,\mathbb{R})$.
Consider the sup norm $\| \cdot \|_\infty$ on $C^k_c(X,\mathbb{R})$ defined by
$$\| f \|_\infty = \sup \{ |f(x_0,\ldots,x_k)| \ | \ (x_0,\ldots,x_k)\in X^{k+1} \}.$$
Let $C^k_{c,b}(X,\mathbb{R})$ be the subspace consisting of continuous bounded $k$-cochains.
Then $C^*_c(X,\mathbb{R})^G$ with the homogeneous coboundary operator becomes a cochain complex. Moreover, the homogeneous coboundary operator on $C^*_c(X,\mathbb{R})^G$ restricts to $C^*_{c,b}(X,\mathbb{R})^G$.

The continuous cohomology $H^*_c(G,\mathbb{R})$ of $G$ is isometrically isomorphic to the cohomology of the cocomplex $C^*_c(X,\mathbb{R})^G$ \cite[Chapter 3]{Gu80}. Similarly the continuous bounded cohomology $H^*_{c,b}(G,\mathbb{R})$ of $G$  is isometrically isomorphic to the cohomology of the subcocomplex $C^*_{c,b}(X,\mathbb{R})^G$ of $C^*_c(X,\mathbb{R})^G$.
Note that the comparison map $c \co H^*_{c,b}(G,\mathbb{R})\rightarrow H^*_c(G,\mathbb{R})$ is realized by the natural inclusion $C^*_{c,b}(X,\mathbb{R})^G \subset C^*_c(X,\mathbb{R})^G$.
Furthermore, for any closed subgroup $\Gamma$ of $G$, two cohomologies $H^*(\Gamma,\mathbb{R})$ and $H^*_b(\Gamma,\mathbb{R})$ are isometrically isomorphic to the cohomologies of the cocomplexes $C^*_c(X,\mathbb{R})^\Gamma$ and  $C^*_{c,b}(X,\mathbb{R})^\Gamma$, respectively. We refer the reader to \cite[Corollary 7.4.10]{Monod} for details.

\subsection{Schl\"{a}fli formula}

The Schl\"{a}fli formula plays a central role in the computation of the volume of hyperbolic polyhedra. This formula relates the variation of the dihedral angles of a smooth family of polyhedra in a space form to the variation of the enclosed volumes.

\begin{theorem}
Let $P_t$ be a smooth one-parameter family of polyhedra in a simply connected $n$-dimensional space of constant curvature $-1$. Then the derivative of the volume of $P_t$ satisfies the equation $$(1-n) d \vol(P_t)= \sum_{F} \vol_{n-2}(F)d\theta_F,$$ where the sum is over all codimension $2$ faces of $P_t$, $\vol_{n-2}$ denotes the $(n-2)$-dimensional volume, and $\theta_F$ denotes the dihedral angle at $F$.
\end{theorem}

Note that the Schl\"{a}fli formula also applies to a smooth one-parameter family of hyperbolic polyhedra with some ideal vertices when $n\geq 4$. In dimension greater than or equal to $4$, the $(n-2)$-dimensional volume of geodesic simplices of codimension $2$ is uniformly bounded from above. Thus the Schl\"{a}fli formula remains valid for ideal hyperbolic polyhedra without any change in the formula.
On the other hand, the Schl\"{a}fli formula is no longer valid in the $3$-dimensional case, since some edge lengths of an ideal hyperbolic polyhedron become infinite. However, the Schl\"{a}fli formula still holds for ideal hyperbolic polyhedra if the edge lengths are measured after removing small horoball neighborhoods of the ideal vertices. For more details, we refer the reader to \cite[Chapter 5]{Hod}.

\section{Reformulation of the volume of representations}\label{sec:volume}

Let $\Gamma$ be a nonuniform lattice in $\SO$. Then $M=\Gamma \backslash \H$ is a complete, noncompact, connected, orientable hyperbolic manifold with finite volume which is homeomorphic to the interior of a compact manifold $\oM$ with boundary. Let $\partial_1 \oM, \ldots, \partial_r \oM$ be the connected components of $\partial \oM$. There is a one-to-one correspondence between the connected components of $\partial \oM$ and the ends of $M$. Let $E_i$ be the end of $M$ associated to $\partial_i \oM$. Each end $E_i$ of $M$ has a neighborhood homeomorphic to $\Gamma_{\xi_i}\backslash U_{\xi_i}$ for some $\xi_i \in \partial \H$ where $U_{\xi_i}$ is a horoball centered at $\xi_i$ and $\Gamma_{\xi_i}$ is the stabilizer of $\xi_i$ in $\Gamma$ for $i=1,\ldots,r$. Then it can be easily seen that
$$\widehat{\widetilde M}=\H \bigcup \cup_{i=1}^r \Gamma \xi_i \text{ and }\wM=\Gamma\backslash \widehat{\widetilde M}.$$
Recall that $\wM$ can be obtained by collapsing each connected component of $\partial \oM$ to a point. From this point of view, we have a quotient map $$\pi : (\oM,\partial \oM) \rightarrow (\wM, \cup_{i=1}^r c_i)$$ where $c_i=\hat{p}(\xi_i)$ for $i=1,\ldots,r$. Obviously $\pi(\partial_i \oM)=c_i$ for all $i$.

Let $X$ be a topological space and $G$ be a group acting on $X$. Let $S_k(X)$ be the abelian group generated by arbitrary $(k+1)$-tuples of points of $X$ modulo $$\langle x_0,\ldots,x_k \rangle = \mathrm{sign}(\tau) \langle x_{\tau(0)},\ldots,x_{\tau(k)} \rangle$$ for any permutation $\tau$ of $\{0,\ldots,k\}$ and $\langle x_0,\ldots,x_k \rangle =0$ if the $x_i$ are not distinct. Define a boundary map $\partial : S_k(X) \rightarrow S_{k-1}(X)$ by $$ \partial \langle x_0,\ldots,x_k \rangle = \sum_{i=0}^k (-1)^i \langle x_0,\ldots,\hat{x}_i,\ldots, x_k \rangle. $$
The $G$-action on $X$ extends to an action on $S_k(X)$: for $g\in G$, $$g\langle x_0,\ldots,x_k \rangle = \langle gx_0,\ldots,gx_k \rangle.$$ Let $S_k(X)_G$ be the free abelian group with coefficients in $\mathbb R$ generated by equivalence classes under the canonical relation associated with the $G$-action on $S_k(X)$. Then a homology $H_*(G,X)$ is defined by $$H_*(G,X):= H_*(S_*(X)_G)$$ and a cohomology $H^*(G,X)$ is defined by $$H^*(G,X):=H^*(\mathrm{Hom}(S_*(X)_G,\mathbb R)),$$ where $\mathrm{Hom}(S_*(X)_G,\mathbb R)$ denotes the real valued homomorphisms on $S_*(X)_G$. This kind of homology was actually introduced by Neumann and Yang \cite[Section 8]{NY} to generalize the Bloch invariants of hyperbolic $3$-manifolds in a homological way. In addition, we define a bounded cohomology $H^*_b(G,X)$ as the cohomology of the subcocomplex $\mathrm{Hom}_b(S_*(X)_G,\mathbb R)$ of bounded real valued homomorphisms on $S_*(X)_G$.

\begin{remark}  It follows by the definition that $$H_*(\pi_1(M),\widetilde{M})=H_*(M) \text{ and } H_*(\pi_1(M),\widehat{\widetilde M})=H_*(\wM)$$ for a manifold $M$. In particular, when $G$ is a locally compact group and $X$ is a locally compact space with continuous proper $G$-action such that $G\backslash X^{k+1}$ is paracompact for all $k\geq 0$, there is a canonical inclusion map $H^*_{c,b}(G,\mathbb R) \hookrightarrow H^*_b(G,X)$ as follows: First, consider the continuous bounded cohomology $H^*_{c,b}(G,\mathbb R)$ as the cohomology of the cocomplex $C^*_{c,b}(X,\mathbb R)^G $. This is possible due to \cite[Theorem 7.4.5]{Monod}. Then in a natural way, continuous bounded $G$-invariant functions on $X^k$ define  homomorphisms in $\mathrm{Hom}(S_k(X)_G,\mathbb R)$. This correspondence induces a cochain map $$C^*_{c,b}(X,\mathbb R)^G  \rightarrow \mathrm{Hom}(S_*(X)_G,\mathbb R).$$
Therefore a canonical inclusion map $H^*_{c,b}(G,\mathbb R) \hookrightarrow H^*_b(G,X)$ is well-defined.
\end{remark}

Let $\rho : \Gamma \rightarrow \SO$ be a representation. Assume that there exists a continuous $\rho$-equivariant map $f: \widehat{\widetilde M} \rightarrow \cH$. We require that the restriction of $f$ to $\widetilde M$ has its image in $\H$, i.e., $f|_{\widetilde M} : \widetilde M \rightarrow \H$. Then we have the following commutative diagram.
\begin{equation}\label{diagram1} \begin{CD} $\xymatrixcolsep{3pc}\xymatrix{
H^n(\SO,\cH) \ar[r]^-{f^*} & H^n(\Gamma, \widehat{\widetilde M}) \ar[rd]^-{\pi^*} \\
H^n_b(\SO,\cH) \ar[r]^-{f^*_b} \ar[d]^-{res_1} \ar[u]_-{c_1} &
H^n_b(\Gamma, \widehat{\widetilde M}) \ar[d]^-{res_2} \ar[rd]^-{\pi^*_b} \ar[u]_-{c_2} & H^n(\oM,\partial \oM)\\
H^n_b(\SO,\H) \ar[r]^-{(f|_{\widetilde M})_b^*} &
H^n_b(\Gamma,\widetilde M) &
H^n_b(\oM,\partial \oM) \ar[l]_-{i^*_b} \ar[u]_-{c_3}}$ \end{CD}\end{equation}

In the above diagram, we use `$c$' to denote comparison maps  and `$res$' for restriction maps.
The map $$i^*_b : H^n_b(\oM,\partial \oM) \rightarrow H^n_b(\Gamma,\widetilde M) = H^n_b(M)$$ is induced from the natural inclusion map $i:(M,\emptyset) \rightarrow (\oM,\partial \oM)$. In fact, $i^*_b$ is an isometric isomorphism between two bounded cohomologies since each connected component of $\partial \oM$ has amenable fundamental group \cite{BBIPP,KK}. The map $\pi^*$ can be understood as follows: $$\pi^* : H^n(\Gamma, \widehat{\widetilde M})= H^n(\widehat{M})\cong H^n(\wM,\cup_{i=1}^r c_i) \rightarrow H^n(\oM, \partial \oM).$$
The map $H^n(\wM,\cup_{i=1}^r c_i) \rightarrow H^n(\oM, \partial \oM)$ is the induced map from $\pi$.
This also works for bounded cohomology. Thus a map $\pi^*_b$ can be defined in the same way as above.

The $\SO$-invariant Riemannian volume form $\omega_{\H}$ on $\H$ gives rise to an $\SO$-invariant continuous function $\vol_n : (\H )^{n+1} \rightarrow \mathbb R$ defined by
\begin{eqnarray}\label{eqn:volcocycle} \vol_n(x_0,\ldots,x_n)=\int_{[x_0,\ldots,x_n]} \omega_{\H},\end{eqnarray} where $[x_0,\ldots,x_n]$ is the geodesic simplex in $\H$ with vertices $x_0,\ldots,x_n$. Since $\vol_n$ is a cocycle in $C^n_c(\H,\mathbb R)^{\SO}$, it determines a continuous cohomology class in $H^n_c(\SO,\mathbb R)$, denoted by $\omega_n^c$. According to the Van Est isomorphism,
$$H^n_c(\SO,\mathbb R) = \mathbb R \cdot \omega_n^c.$$
The cocycle $\vol_n$ is actually bounded and thus it determines a continuous bounded cohomology class $\omega_n^{c,b} \in H^n_{c,b}(\SO,\mathbb R)$.

Denote by $\omega_n^b \in H^n_b(\SO,\H)$ the image of $\omega_n^{c,b}$ via the canonical inclusion $H^n_{c,b}(\SO,\mathbb R) \hookrightarrow H^n_b(\SO,\H)$. The composition of $(f|_{\widetilde M})_b^*$ with the canonical inclusion, $$H^n_{c,b}(\SO,\mathbb R) \rightarrow H^n_{b}(\SO, \H),$$
is canonically identified with the map $\rho^*_b : H^n_{c,b}(\SO,\mathbb R) \rightarrow  H^n_b(\Gamma, \widetilde M)=H^n_b(M,\mathbb R)=H^n_b(\Gamma,\mathbb R)$. This implies that $(f|_{\widetilde M})_b^* (\omega^b_n)=\rho^*_b (\omega^{c,b}_n)$ and the map $(f|_{\widetilde M})_b^*$ does not depend on the choice of $f$ but only on the representation $\rho$.

\begin{definition}
The volume $\vol(\rho)$ of a representation $\rho:\Gamma \rightarrow \SO$ is defined as $$\vol(\rho) = \langle (c_3 \circ (i^*_b)^{-1} \circ (f|_{\widetilde M})^*_b) (\omega_n^b), [\oM,\partial \oM] \rangle.$$
\end{definition}

Note that the above definition works because $i^*_b$ is an isometric isomorphism.
Bucher--Burger--Iozzi \cite{BBI} gave a complete proof of volume rigidity from the point of view of bounded cohomology.
\begin{theorem}[Bucher--Burger--Iozzi, \cite{BBI}]\label{BBI} Let $n\geq 3$. Let $i : \Gamma \hookrightarrow \SO$ be a lattice embedding and let $\rho : \Gamma \rightarrow \SO$ be any representation. Then $$|\vol(\rho)|\leq |\vol(i)|=\vol(M),$$ with equality if and only if $\rho$ is conjugate to $i$ by an isometry.\end{theorem}

Now we will reformulate the volume of $\rho$ with the commutative diagram (\ref{diagram1}).
First observe that the cocycle $\vol_n : (\H)^{n+1} \rightarrow \mathbb R$ can be extended to an $\SO$-invariant continuous function $\overline{\vol}_n : \left(\cH \right)^{n+1} \rightarrow \mathbb R$ by using the same formula (\ref{eqn:volcocycle}). This is possible because a geodesic simplex with vertices on $\cH$ is well-defined. Furthermore, the volume of the geodesic simplices with vertices on $\cH$ is uniformly bounded from above.
Hence, $\overline{\vol}_n$ defines both a cohomology class $\overline{\omega}_n \in H^n(\SO,\cH)$ and a bounded cohomology class $\overline{\omega}_n^b \in H^n_b(\SO,\cH)$. Notice that the restriction of $\overline{\vol}_n$ to $(\H)^{n+1}$ is $\vol_n$ and so, $res_1(\overline{\omega}_n^b)=\omega_n^b$. From the commutativity of the diagram (\ref{diagram1}), we have
\begin{eqnarray*} (c_3 \circ (i^*_b)^{-1} \circ (f|_{\widetilde M})^*_b )({\omega}_n^b)&=&(c_3 \circ (i^*_b)^{-1} \circ (f|_{\widetilde M})^*_b \circ res_1)(\overline{\omega}_n^b)\\
&=&(\pi^* \circ  f^* \circ c_1 )(\overline{\omega}_n^b)\\&=&(\pi^* \circ f^* )(\overline{\omega}_n).
\end{eqnarray*}
This implies that
\begin{eqnarray*} \vol(\rho) &=& \langle (\pi^* \circ f^* )(\overline{\omega}_n), [\oM,\partial \oM] \rangle \\ &=& \langle f^* (\overline{\omega}_n), \pi_*[\oM,\partial \oM] \rangle = \langle f^* (\overline{\omega}_n), [\widehat M] \rangle. \end{eqnarray*}
Here, note that the value $\langle f^* \overline{\omega}_n, [\widehat M] \rangle$ is independent of the choice of $f$ because $\vol(\rho)$ depends only on the representation $\rho :\Gamma \rightarrow \SO$. In fact, it can be seen that for a pseudo-developing map $D_\rho : \widehat{\widetilde M} \rightarrow \cH$, $$\vol(\rho)=\langle  D_\rho^* (\overline{\omega}_n), [\widehat M] \rangle= \int_M D_\rho^* \omega_{\H}.$$
This verifies that two different definitions of Francaviglia--Klaff \cite{FK} and Bucher--Burger--Iozzi \cite{BBI} for the volume of a representation of a nonuniform lattice in fact give the same value.

\section{Construction of a fundamental cycle}\label{sec:triangulation}

In the previous section, we reformulated the volume of a representation in terms of a continuous equivariant map and the fundamental class of $\wM$. If we have a `good' equivariant map and fundamental cycle, we can easily measure the variation of values of the volume of a representation over $\mathrm{Hom}(\Gamma,\SO)$. In this section, we construct a certain good fundamental cycle of $\wM$.

Recall that $\widehat{\widetilde M}= \H \cup \cup_{i=1}^r \Gamma \xi_i $ and $\wM = \Gamma \backslash \left( \H \cup \cup_{i=1}^r \Gamma \xi_i \right)$. We define a submanifold $M_0$ of $M$ by $$M_0 = \Gamma \backslash \left( \H - \cup_{i=1}^r \Gamma U_{\xi_i} \right).$$
Then $M_0$ is a compact core of $M$ whose boundary consists of projections of horospheres in $\H$.
Thus $M_0$ is a differentiable manifold with boundary. By the work of Munkres \cite{Mu}, there is a triangulation of $M_0$. In fact, any $ C^k$ triangulation of the boundary can be extended to a $C^k$ triangulation of the whole manifold.

Let $\mathcal T_{M_0}$ be a triangulation of $M_0$. It induces a triangulation of $\partial M_0$, denoted by  $\T_{\partial M_0}$. Then $\T_{M_0}$ is lifted to a triangulation $\mathcal T_{\widetilde M_0}$ of $\widetilde M_0$ and $\T_{\partial M_0}$ is lifted to a triangulation $\mathcal T_{\partial \widetilde M_0}$ of $\partial \widetilde M_0$. Note that $$\widetilde M_0 =\H - \cup_{i=1}^r \Gamma U_{\xi_i} \text{ and } \partial \widetilde M_0 = \cup_{i=1}^r \Gamma \partial U_{\xi_i}.$$

If $\tau$ is an $(n-1)$-simplex on a horosphere $\gamma \partial U_{\xi_i}$ for some $\gamma\in \Gamma$, then take the geodesic cone on $\tau$ with the top point $\gamma \xi_i$. Denote it by $\mathrm{Cone}(\tau)$. Proceed in this way for all the simplices in the triangulation $\T_{\partial \widetilde M_0}$ of $\partial \widetilde M_0$. Then
we obtain a triangulation of $\cup_{i=1}^r \Gamma (U_{\xi_i} \cup \xi_i )$, denoted by $\T_{ \mathrm{Cone} (\partial \widetilde M_0)}$. Finally, two $\Gamma$-equivariant triangulations $\T_{\widetilde M_0}$ and $\T_{ \mathrm{Cone} (\partial \widetilde M_0)}$ induce a $\Gamma$-equivariant triangulation $\T$ of $\widehat{\widetilde M}$ and this triangulation descends to a triangulation of $\wM$, denoted by $\T_{\wM}$.

The triangulation $\T_{\wM}$ gives rise to a fundamental cycle $\widehat z$ of $\widehat{M}$. Note that a simplex occurring in $\widehat z$ comes from either a simplex of $\T_{\widetilde M_0}$ or a simplex of $\T_{ \mathrm{Cone} (\partial \widetilde M_0)}$. Thus  $\widehat z$ can be expressed as $$\widehat z = \sum_{\bar s \in \T_{M_0}} \bar s +\sum_{\bar \tau \in \T_{\partial M_0}} \mathrm{Cone}(\bar \tau)$$
where $\mathrm{Cone}(\bar \tau)=\hat p (\mathrm{Cone}(\tau))$ for a lift $\tau$ of $\bar \tau$ to $\widehat{\widetilde M}$. It is easy to see that $\mathrm{Cone}(\bar \tau)$ is independent of the choice of a lift $\tau$.

Let $s$ and $\tau$ denote lifts of $\bar s$ and $\bar \tau$ to $\widehat{\widetilde M}$. Using the fundamental cycle $\widehat z$ of $\wM$, the volume of $\rho$ can be computed as follows:
\begin{eqnarray*}\label{voleqn} \lefteqn{\vol(\rho)= \sum_{\bar s \in \T_{M_0}}  \vol_n(f(s(e_0)),\ldots,f(s(e_n)))} \\  & & \ \ \ \ \ \ \ \ \ +\sum_{\bar \tau \in \T_{\partial M_0}}  \vol_n(f(\mathrm{Cone}(\tau)(e_0)),\ldots,f(\mathrm{Cone}(\tau)(e_n))) \end{eqnarray*} where $e_0,\ldots,e_n$ are the vertices of the standard $n$-simplex $\Delta^n$.
Since the restriction $f|_{\widetilde M}$ maps $\widetilde M$ to $\H$, $[f(s(e_0)),\ldots,f(s(e_n))]$ is a geodesic simplex in $\H$ and $[f(\mathrm{Cone}(\tau)(e_0)),\ldots,f(\mathrm{Cone}(\tau)(e_n))]$ is an (ideal) geodesic simplex with at most one ideal vertex.

\section{Construction of an equivariant map}\label{sec:equimap}

In this section, we will construct a $\rho$-equivariant map $f: \widehat{\widetilde M} \rightarrow \cH$ whose restriction to $\widetilde M$ has its image in $\H$. We start by recalling the following lemma due to Zickert \cite{Zi}.

\begin{lemma}\label{uniquemap} Let $\Delta$ be an $n$-simplex in $\mathbb R^n$ with an ordering of the vertices. Given any geodesic $n$-simplex $\Delta' \in \cH$ with a vertex ordering, there exists a unique homeomorphism from $\Delta$ to $\Delta'$ which restricts to an order-preserving map of vertices and takes Euclidean straight lines to hyperbolic straight lines.\end{lemma}

Lemma \ref{uniquemap} can be seen if one works in the Klein model of hyperbolic space where the hyperbolic geodesic lines and the Euclidean straight lines coincide. For more details, see \cite[Section 4]{Zi}.

\begin{lemma}\label{lem:equimap} There exists a triangulation $\T$ of $\widehat{\widetilde M}$ and a continuous, $\rho$-equivariant map $f: \widehat{\widetilde M} \rightarrow \cH$ that is nondegenerate in the sense that the image of any simplex of $\T$ is a nondegenerate, geodesic (ideal) simplex. Furthermore, the restriction of $f$ to $\widetilde M$ has its image in $\H$.
\end{lemma}

\begin{proof}
Our construction basically follows that of Besson--Courtois--Gallot \cite[Section 5]{BCG} in the uniform lattice case. We first build a $\Gamma$-equivariant triangulation $\mathcal T$ of $\widehat{\widetilde M}$ and then define $f$ on each simplex of $\mathcal T$.

We stick to the notation of Section \ref{sec:triangulation}. As described in the previous section, if a triangulation $\T_{M_0}$ on the compact core $M_0$ is given, we obtain the $\Gamma$-equivariant triangulation $\T$ on $\widehat{\widetilde M}$ and the triangulation $\T_{\wM}$ on $\wM$ induced from $\T$.
By Lemma \ref{uniquemap}, it can be easily seen that $f$ is uniquely determined by its value on the vertices of $\T$. Hence we will construct $f$ by choosing its value on the vertices of $\T$.  Note that the vertices of $\T$ are the union of the vertices of $\T_{\widetilde M_0}$ and $\cup_{i=1}^r \Gamma \xi_i$.
We first choose the value of $f$ on the vertices of $\T_{\widetilde M_0}$ and then on $\cup_{i=1}^r \Gamma \xi_i$.

Choose a sufficiently fine triangulation $\mathcal T_{M_0}$ of $M_0$. Let $D_0$ be a fundamental domain of $M_0$ in $\widetilde M_0$ and $\{v_1,\ldots,v_N\}$ be the set of vertices of $\mathcal T_{\widetilde M_0}$ contained in $D_0$. We can choose $D_0$ in such a way that $v_i \neq \gamma v_j$ for all $\gamma \in \Gamma$ and $i\neq j$. Recall that the vertex set of any simplex of $\T$ is either of the form $$\{\gamma_{i_0} v_{i_0},\ldots,\gamma_{i_n}v_{i_n}  \}$$ or of the form $$\{ \gamma_{i_0}v_{i_0},\ldots,\gamma_{i_{n-1}}v_{i_{n-1}},\gamma_{i_n}\xi_{i_n} \}$$
for $\gamma_{i_0},\ldots, \gamma_{i_n} \in \Gamma$ and $\xi_{i_n} \in \{\xi_1,\ldots, \xi_r\}$. Note that for any $\gamma\in \Gamma$, $\gamma\gamma_{i_k}v_{i_k}\neq \gamma_{i_l}v_{i_l}$ for any $k\neq l$. 

Following the construction of Besson--Courtois--Gallot in \cite{BCG}, one can choose $y_1,\ldots,y_N \in \H$ such that if $\{\gamma_{i_0} v_{i_0},\ldots,\gamma_{i_n}v_{i_n}  \}$ is the vertex set of a simplex of $\mathcal T_{\widetilde M_0}$, the geodesic simplex $[\rho(\gamma_{i_0})y_{i_0},\ldots,\rho(\gamma_{i_n})y_{i_n}]$ is nondegenerate.
Furthermore there exists a small neighborhood $V_i$ of $y_i$ for each $i=1,\ldots,N$ such that for any choice $y_i' \in V_i$, $(y_1',\ldots,y_N')$, the nondegeneracy property still holds. See \cite[Lemma 5.2]{BCG} for a detailed proof.

Now it remains to choose the value of $f$ on $\cup_{i=1}^r\Gamma \xi_i$.
By the Margulis lemma, the stabilizer subgroup $\Gamma_{\xi_i}$ is an almost nilpotent group for all $i$. Since any amenable subgroup of $\SO$ fixes a point in $\cH$ \cite{Mo},  $\rho(\Gamma_{\xi_i})$ has at least one fixed point in $\cH$. Choose a fixed point $\eta_i$ of $\rho(\Gamma_{\xi_i})$ in $\cH$ for each $i$.
We set the value of $f$ on the vertices of $\T$ in $\Gamma$-equivariant way as follows:
\begin{itemize}
\item[-] $f(v_i)= y_i$ for $i=1,\ldots,N$ and $f(\gamma v_i)=\rho(\gamma) y_i$ for all $\gamma\in \Gamma$,
\item[-] $f(\xi_i)=\eta_i$ for $i=1,\ldots,r$ and $f(\gamma \xi_i)=\rho(\gamma) \eta_i$ for all $\gamma\in \Gamma$.
\end{itemize}

We only need to check that if $\{ \gamma_{i_0}v_{i_0},\ldots,\gamma_{i_{n-1}}v_{i_{n-1}},\gamma_{i_n}\xi_{i_n} \}$ is the vertex set of a simplex of $\T_{ \mathrm{Cone} (\partial \widetilde M_0)}$, then $[ \rho(\gamma_{i_0})y_{i_0},\ldots,\rho(\gamma_{i_{n-1}})y_{i_{n-1}},\rho(\gamma_{i_n})\eta_{i_n} ]$ is a nondegenerate (ideal) geodesic simplex. Due to the $\Gamma$-equivariance of $f$, it suffices to show this for simplices of $\T$ with the vertex set of the form $\{ v_{i_0},\gamma_{i_1}v_{i_1},\ldots,\gamma_{i_{n-1}}v_{i_{n-1}},\gamma_{i_n}\xi_{i_n} \}$.
Since a simplex of $\T_{ \mathrm{Cone} (\partial \widetilde M_0)}$ is the geodesic cone on an $(n-1)$-simplex of $\T_{\partial \widetilde M_0}$ and every simplex of $\T_{\widetilde M_0}$ is mapped to a nondegenerate geodesic simplex by $f$, $[ y_{i_0},\rho(\gamma_{i_1})y_{i_1},\ldots,\rho(\gamma_{i_{n-1}})y_{i_{n-1}}]$ is a nondegenerate geodesic $(n-1)$-simplex in $\H$.

Assume that $[ y_{i_0},\rho(\gamma_{i_1})y_{i_1},\ldots,\rho(\gamma_{i_{n-1}})y_{i_{n-1}},\rho(\gamma_{i_n})\eta_{i_n} ]$ is degenerate.
We define a set $Q$ by $$Q=\{y\in \H \ | \ [y,\rho(\gamma_{i_1})y_{i_1},\ldots,\rho(\gamma_{i_{n-1}})y_{i_{n-1}},\rho(\gamma_{i_n})\eta_{i_n} ] \text{ is degenerate} \}.$$
Note that $Q$ is an $(n-1)$-hyperbolic space in $\H$ containing $y_{i_0}$. Since $V_{i_0} \cap Q$ is a codimension $1$ subset of $V_{i_0}$, it is possible to move $y_{i_0}$ to a point in $V_{i_0}-Q$ so that $[ y_{i_0},\rho(\gamma_{i_1})y_{i_1},\ldots,\rho(\gamma_{i_{n-1}})y_{i_{n-1}},\rho(\gamma_{i_n})\eta_{i_n} ]$ is a nondegenerate geodesic simplex. By proceeding in this way for all simplices of $\T$ with vertex $v_{i_0}$, we can choose $y_{i_0}$ such that all simplices of $\T$ with vertex $v_{i_0}$ are mapped to nondegenerate geodesic simplices by $f$. Proceeding by induction on the vertex set $\{v_1,\ldots,v_N\}$, we obtain the desired map $f: \widehat{\widetilde M} \rightarrow \cH$.
\end{proof}

Note that by the construction of $f$, the restriction of $f$ to $\widetilde M$ has its image in $\H$.
Hence we can apply $f$ to the reformulation of the volume of a representation in Section \ref{sec:volume}.

\section{One cusped hyperbolic manifold case}\label{sec:onecusp}

For the sake of simplicity, we first work with one cusped hyperbolic manifolds in this section.
Throughout this section, $\Gamma_1$ denotes a nonuniform lattice in $\SO$ such that $M_1=\Gamma_1 \backslash \H$ is a one cusped hyperbolic manifold.

Let $\rho : \Gamma_1 \rightarrow \SO$ be a representation. The (unique) end of $M_1$ has a neighborhood homeomorphic to $U_\xi/\Gamma_{1,\xi}$ for some $\xi \in \partial \H$ where $\Gamma_{1,\xi}$ is the stabilizer subgroup of $\xi$ in $\Gamma_1$.
Since $\rho(\Gamma_{1,\xi})$ is amenable, it is contained in either a maximal compact subgroup or minimal parabolic subgroup of $\SO$ \cite{Mo}.
Note that a maximal compact subgroup is the stabilizer subgroup of a point of $\H$ in $\SO$, and a minimal parabolic subgroup is the stabilizer subgroup of an ideal point of $\partial \H$ in $\SO$.

Denote by $\hom_K(\Gamma_1,\SO)$ (resp. $\hom_P(\Gamma_1,\SO)$) the subset of $\hom(\Gamma_1,\SO)$ consisting of representations which send $\Gamma_{1,\xi}$ into a maximal compact subgroup (resp. a minimal parabolic subgroup) of $\SO$. These subsets are well-defined, independently of the choice of $\Gamma_{1,\xi}$.
Then it is obvious that $$\hom(\Gamma_1,\SO)=\hom_K(\Gamma_1,\SO) \cup \hom_P(\Gamma_1,\SO).$$ Note that if $\rho \in \hom_K(\Gamma_1,\SO)$, then $\rho(\Gamma_{1,\xi})$ has a fixed point on $\H$. If $\rho \in \hom_P(\Gamma_1,\SO)$, then $\rho(\Gamma_{1,\xi})$ has a fixed point on $\partial \H$.
Note that $\hom_K(\Gamma_1,\SO)$ and $\hom_P(\Gamma_1,\SO)$ can have nonempty intersection. If $G=KAN$ and $P=MAN$, then a representation
whose image is contained in $M$ belong to both of them.
\begin{lemma}\label{lem:closedset} $\hom_P(\Gamma_1,\SO)$ is a closed subset of $\hom(\Gamma_1,\SO)$.\end{lemma}
\begin{proof}
It suffices to show that every limit of a sequence in $\hom_P(\Gamma_1,\SO)$ is contained in $\hom_P(\Gamma_1,\SO)$. First note that in the case of $\SO$, all minimal parabolic subgroups are closed and conjugate to each other by an inner automorphism. Furthermore, since the normalizer of a minimal parabolic subgroup is itself, the space of minimal parabolic subgroups is identified with $\SO/P$ for a minimal parabolic subgroup $P$ of $\SO$. In fact, using the Cartan decomposition of $\SO$, it is not difficult to see that $$\SO/P \cong \mathrm{SO}(n) / \mathrm{SO}(n-1) \cong S^{n-1},$$
where $S^{n-1}$ denotes the $(n-1)$ unit sphere.

Let $(\rho_i)$ be a sequence in $\hom_P(\Gamma_1,\SO)$ converging to a representation $\rho$.
Since $\rho_i \in \hom_P(\Gamma_1,\SO)$, there exists a minimal parabolic subgroup $P_i$ such that $\rho_i(\Gamma_{1,\xi}) \subset P_i$ for each $i \in \mathbb N$. Since minimal parabolic subgroups are conjugate to each other, we can set $P_i=k_i P k_i^{-1}$ for some $k_i \in \mathrm{SO}(n)$.
For each $\alpha \in \Gamma_{1,\xi}$, let $\rho_i(\alpha)=k_i p_i^\alpha k_i^{-1}$ for some $p_i^\alpha \in P$.
Here $k_i$ depends on $\rho_i$ but $p_i^\alpha$ depends on both $\alpha$ and $\rho_i$.
Then by the assumption, the sequence $(\rho_i(\alpha))$ converges to the element $\rho(\alpha)$ of $\SO$. Since $k_i$ is in the compact subgroup $\mathrm{SO}(n)$ of $\SO$, the sequence $(p_i^\alpha)$ has to converge for all $\alpha \in \Gamma_{1,\xi}$. Moreover, $k_i$ converges.
Note that the limit points of $(p_i^\alpha)$ and $(k_i)$ are contained in $P$ and $\mathrm{SO}(n)$ respectively because both $P$ and $\mathrm{SO}(n)$ are closed. This implies that for all $\alpha \in \Gamma_{1,\xi}$,
$$\rho(\alpha) \in k P k^{-1},$$ where $k$ denotes the limit of $(k_i)$.
Therefore $\rho \in \hom_P(\Gamma_1,\SO)$.
\end{proof}



We define an open subset of $\hom(\Gamma_1,\SO)$ by $$\mathcal O = \hom(\Gamma_1,\SO)-\hom_P(\Gamma_1,\SO).$$
It is easy to see that $\O \cup \hom_P(\Gamma_1,\SO) = \hom(\Gamma_1,\SO)$ and $\O \subset \hom_K(\Gamma_1,\SO)$.
The representation $\rho$ of $\mathcal O$ has the following property on the set of fixed points of $\rho(\Gamma_{1,\xi})$.

\begin{lemma}\label{lem:fixedpoint} Let $\rho$ be a representation in $\O$. Then $\rho(\Gamma_{1,\xi})$ has a unique fixed point on $\H$.
\end{lemma}
\begin{proof}
By the definition of $\mathcal O$, $\rho(\Gamma_{1,\xi})$ can not have a fixed point on $\partial \H$.
Hence it suffices to show that $\rho(\Gamma_{1,\xi})$ does not have two fixed points on $\H$.
If $\rho(\Gamma_{1,\xi})$ has two fixed points on $\H$, the geodesic connecting them is also fixed by $\rho(\Gamma_{1,\xi})$. Then $\rho(\Gamma_{1,\xi})$ must fix two ideal points on $\partial \H$ determined by the geodesic.
This contradicts the fact that $\rho(\Gamma_{1,\xi})$ does not fix any point on $\partial \H$. Therefore $\rho(\Gamma_{1,\xi})$ has a unique fixed point on $\H$.
\end{proof}

We define two kinds of paths on $\hom(\Gamma_1,\SO)$ depending on the behavior of a peripheral subgroup of $\Gamma_1$ along the path.

\begin{definition} A path on $\hom(\Gamma_1,\SO)$ is said to be a \emph{boundary-compact path} if any representation on the path sends a peripheral subgroup of $\Gamma_1$ into a maximal compact subgroup of $\SO$. A path on $\hom(\Gamma_1,\SO)$ is said to be a \emph{boundary-parabolic path} if any representation on the path sends a peripheral subgroup of $\Gamma_1$ into a minimal parabolic subgroup of $\SO$.
\end{definition}

For instance, a path on $\hom_K(\Gamma_1,\SO)$ is a boundary-compact path and a path on $\hom_P(\Gamma_1,\SO)$ is a boundary-parabolic path.

\begin{lemma}\label{lem:equipath}
Let $\rho_t : \Gamma_1 \rightarrow \SO$ be a $C^1$-smooth path parameterized by $t\in [0,1]$. If $\rho_t$ is either a boundary-compact path or boundary-parabolic path, there is a $C^1$-smooth path $\eta : [0,1] \rightarrow \cH$ such that $\eta(t)$ is a fixed point of $\rho_t(\Gamma_{1,\xi})$ for all $t\in [0,1]$. Furthermore, if $\rho_t$ is a boundary-compact path, $\eta$ has its image in $\H$ and, if $\rho_t$ is a boundary-parabolic path, $\eta$ has its image in $\partial \H$.
\end{lemma}

\begin{proof}
Suppose that $\rho_t$ is a boundary-compact path. Since all maximal compact subgroups are conjugate to each other, we can set $\rho_t(\Gamma_{1,\xi}) \subset g_t \mathrm{SO}(n) g_t^{-1}$ for some $g_t \in \SO$. Furthermore, one may choose $g_t$ to be a $C^1$-smooth path on $\SO$ since $\rho_t$ is a $C^1$-smooth path. Let $o$ denote the point on $\H$ corresponding to the fixed point of $\mathrm{SO}(n)\subset \SO$. Then define a path $\eta :[0,1]\rightarrow \H$ by $\eta(t) = g_t o.$ Obviously, $\eta$ is a $C^1$-smooth path on $\H$ and $\eta(t)=g_t o$ is a fixed point of $\rho_t(\Gamma_{1,\xi})$. This is the desired path.

If $\rho_t$ is a boundary-parabolic path, similarly to the above, we can find a $C^1$-smooth path $k_t$ on $\mathrm{SO}(n)$ with $\rho_t(\Gamma_{1,\xi}) \subset k_t P k_t^{-1}$ for a fixed minimal parabolic subgroup $P$ of $\SO$. Let $z_0 \in \partial \H$ denote the fixed point of $P$. Then we can define a path $\eta : [0,1] \rightarrow \partial \H$ by $\eta(t) = k_t z_0$. Clearly, it is a $C^1$-smooth path and $\eta(t)$ is a fixed point of $\rho_t(\Gamma_{1,\xi})$. This completes the proof.
\end{proof}

Let $\rho_t : \Gamma_1 \rightarrow \SO$ be a $C^1$-smooth path parameterized by $t\in [0,1]$. According to Lemma \ref{lem:equimap}, there exist a triangulation $\T$ on $\widehat{\widetilde M_1}$ and a continuous, nondegenerate, $\rho_0$-equivariant map $f_0 : \widehat{\widetilde M_1} \rightarrow \cH$ with $f_0(\xi)=\eta(0)$. Then we define a one-parameter family $f_t :\widehat{\widetilde M_1} \rightarrow \cH$ as follows:
\begin{itemize}
\item[-] $f_t(v_i)=y_i$ and $f_t(\gamma v_i)=\rho_t(\gamma) y_i$,
\item[-] $f_t(\xi)=\eta(t)$ and $f_t(\gamma \xi)=\rho_t(\gamma) \eta(t)$,
\end{itemize}
for all $\gamma \in \Gamma_1$ and $i=1,\ldots,N$.
By the construction of $f_t$, it is clear that $f_t$ is $\rho_t$-equivariant and $(f_t)$ is a  one-parameter family of maps from $\widehat{\widetilde M_1}$ to $\cH$. Moreover, since any simplex of $\mathcal T$ is mapped to a nondegenerate geodesic simplex by $f_0$, any simplex of $\mathcal T$ is also mapped to a nondegenerate geodesic simplex by $f_t$ for all sufficiently small $t$.

\begin{remark}
For a one-parameter family $f_t$ as above, we get  a smooth one-parameter family $f_t(s)$ of simplices in $\cH$ for each simplex $s \in \T$.
If $s$ does not have an ideal vertex, then $f_t(s)$ is a smooth one-parameter family of simplices in $\H$. However if $s$ has an ideal vertex, then $f_t(s)$ is either a simplex or an ideal simplex depending on the image of the ideal vertex of $s$ by $f_t$. In order to apply  Schl\"{a}fli formula to $f_t(s)$, it is required that all $f_t(s)$  are either simplices or ideal simplices. Lemma \ref{lem:equipath} implies that the required condition is satisfied if $\rho_t$ is either a boundary-compact path or boundary-parabolic path.
\end{remark}

Let $\T_{\cH}$ be the collection of geodesic simplices of $\cH$ obtained by $f_0$ and $\T$. Let $F$ be a face of codimension $2$ of $\T$ and denote by $F'$ its image under $f_0$. Note that $F'$ is a face of codimension $2$ of $\T_{\cH}$. The star of $F$ contains a finite number of simplices $s_1,\ldots,s_k$ of $\T$. Denote their images by $s_1',\ldots,s_k'$.
Since the restriction map of $f_0$ to each $s_i$ is a homeomorphism onto $s_i'$ for all $i$, there are small tubular neighborhoods $\Tub (F)$ and $\Tub (F')$ of $F$ and $F'$ such that $\partial \Tub (F)$ projects onto $\partial \Tub (F')$ by $f_0$. Hence $f_0$ defines a homomorphism $$(f_0)_* : H_1( \partial \Tub(F),\mathbb Z) \rightarrow H_1( \partial \Tub(F'),\mathbb Z).$$

Noting that $\Tub (F)$ and $\Tub (F')$ are homeomorphic to $F \times D^2$, one can easily see that $H_1( \partial \Tub(F),\mathbb Z) \cong \mathbb Z$ and $H_1( \partial \Tub(F'),\mathbb Z) \cong \mathbb Z$.
Choose an orientation on $\H$, $F$ and $F'$.
Choose a generator $g_F$ of $H_1( \partial \Tub(F),\mathbb Z)$ in such a way that the orientations on $F$ and $g_F$ are compatible with the orientation on $\H$. In the same way, choose a generator $g_{F'}$ of $H_1( \partial \Tub(F'),\mathbb Z)$. Define the transverse degree of $f_0$ to $F$, denoted by $\mathrm{deg}_F f_0 \in \mathbb Z$, by $$ (f_0)_* g_F = (\mathrm{deg}_F f_0) \cdot g_{F'}.$$

The orientation on $\H$ induces an orientation on each simplex of $\T$ and $\T_{\cH}$.
For a simplex $s$ of $\T$, let $s'=f_0(s)$. We define
\begin{displaymath}
\epsilon(s')=\left\{ \begin{array}{ll} +1 & \textrm{if $f_0 : s \rightarrow s'$ is orientation-preserving} \\
-1 & \textrm{otherwise} \end{array} \right. \end{displaymath}
Then it is not difficult to see that
$$ 2\pi \mathrm{deg}_F f_0 = \pm \sum_{s / F\subset s} \epsilon(s') \theta(F',s'),$$
where $\theta(F',s')$ denotes the dihedral angle of $s'$ at $F'$ and the sum is taken over all simplices of $\T$ containing the face $F$. See \cite[Section 5]{BCG} for a detailed proof of this.

Let $F'(t)=f_t(F)$ and $s'(t)=f_t(s)$. Recall that the vertex set of a simplex $s$ of $\T$ is either $\{\gamma_{i_0} v_{i_0},\ldots,\gamma_{i_n}v_{i_n}  \}$ or $\{ \gamma_{i_0}v_{i_0},\ldots,\gamma_{i_{n-1}}v_{i_{n-1}},\gamma_{i_n}\xi\}$ for some $\gamma_{i_0},\ldots, \gamma_{i_n} \in \Gamma$.
Hence $s'(t)$ can be written as $$s'(t)=[\rho_t(\gamma_{i_0}) y_{i_0},\ldots,\rho_t(\gamma_{i_{n-1}})y_{i_{n-1}} ,\rho_t(\gamma_{i_n})z_{i_n} ]$$ for some $z_{i_n} \in \{y_1,\ldots,y_N,\eta(t) \}$.

In the case that $\rho_t$ is a boundary-compact path, $s'(t)$ is a $C^1$-smooth one-parameter family of geodesic simplices in $\H$ for all $s \in \T$ due to $\eta(t) \in \H$. In the case that $\rho_t$ is a boundary-parabolic path, $s'(t)$ is either a
$C^1$-smooth one-parameter family of geodesic simplices in $\H$ or ideal geodesic simplices with all vertices in $\H$ but one vertex in $\partial \H$.
In either case, it is possible to apply the Schl\"{a}fli formula to a $C^1$-smooth one-parameter family $s'(t)$ of (ideal) geodesic simplices for all $s \in \T$.

Since $\partial \Tub (F'(t))$ moves in a $C^1$-manner in $\cH$ and are all homeomorphic to each other,
the transverse degree of $f_t$ to $F$ does not change. Thus we have $\mathrm{deg}_F f_t = \mathrm{deg}_F f_0$. Similarly, the orientation of $f_t$ on $s$ does not change for each $s \in \T$ i.e., $\epsilon(s')=\epsilon(s'(t))$ for all $t\in [0,1]$. These facts imply that $$\frac{d (\mathrm{deg}_F f_t)}{dt}\bigg|_{t=0}= \sum_{s / F\subset s} \epsilon(s') \frac{d \theta(t;F',s')}{dt}\bigg|_{t=0}  =0,$$
where $\theta(t;F',s')$ denotes the dihedral angle of $s'(t)$ at $F'(t)$.

\begin{theorem}\label{thm:onecuspedcase} Let $n\geq 4$ and $\rho_t : \Gamma_1 \rightarrow \SO$ be a $C^1$-smooth path on $\hom(\Gamma_1,\SO)$. Then $\vol(\rho_t)$ is constant.
\end{theorem}

\begin{proof}
First of all, we suppose that $\rho_t$ is either a boundary-compact or boundary-parabolic path on $\hom(\Gamma_1,\SO)$.
Let $\bar s$ denote a simplex of $\T_{\widehat M_1}$ obtained by a simplex $s \in \T$. Then
\begin{eqnarray*}
\vol(\rho_t) &=& \sum_{\bar s \in \T_{\widehat M_1}} \vol_n((f_t\circ s)(e_0),\ldots,(f_t\circ s)(e_n)) \\ &=& \sum_{\bar s \in \T_{\widehat M_1}} \epsilon(s') \vol(s'(t)) =: \sum_{\bar s \in \T_{\widehat M_1}} \epsilon(\bar s) \vol(\bar s,t) \end{eqnarray*}
where $\epsilon(\bar s)= \epsilon(s')$ and $\vol(\bar s,t)=\vol(s'(t))$. Note that $\epsilon(\bar s)$ and $\vol(\bar s,t)$ are independent of the choice of $s$.

In the $n\geq 4$ case, the $(n-2)$-dimensional volume of geodesic $(n-2)$-simplices is well-defined and actually uniformly bounded from above. Applying the Schl\"{a}fli formula to each one-parameter family $s'(t)$ in the above formula, we have
\begin{eqnarray*}
\frac{d \vol(\rho_t)}{dt} \bigg|_{t=0} &=& \sum_{\bar s \in \T_{\widehat M_1}}  \epsilon(\bar s) \frac{d \vol(\bar s, t)}{dt} \bigg|_{t=0} \\ &=&  \sum_{\bar s \in \T_{\widehat M_1}} \epsilon(\bar s)   \frac{1}{1-n}  \sum_{\bar F / \bar F \subset \bar s} \vol_{n-2}(\bar F, 0) \frac{d \theta (t;\bar F, \bar s)}{dt} \bigg|_{t=0}\\
&=& \frac{1}{1-n} \sum_{\bar F} \left( \sum_{\bar s / \bar F \subset \bar s} \epsilon(\bar s)  \frac{d \theta (t;\bar F, \bar s)}{dt}   \bigg|_{t=0} \right) \vol_{n-2}(\bar F, 0) \\
&=& 0
\end{eqnarray*}
where $\vol_{n-2}(\bar F, t)$ is the $(n-2)$-dimensional volume of $\bar F(t)$ and $\theta (t;\bar F, \bar s)=\theta(t;F',s')$. Since this computation of the derivative of $\vol(\rho_t)$ at $t=0$ works for any $t\in [0,1]$, we can conclude that the derivative of the volume of a representation on either a boundary-compact path or boundary-parabolic path vanishes.

Let $\rho_t : \Gamma_1 \rightarrow \SO$ be an arbitrary $C^1$-smooth path parameterized by $t\in[0,1]$. Then we define
$$I_{\mathcal O}= \{ t\in (0,1) \ | \ \rho_t \in \mathcal O \}.$$
Since $\mathcal O$ is an open subset of $\hom(\Gamma_1,\SO)$, $I_{\mathcal O}$ is an open subset of $(0,1)$ too, and so it is actually a countable union of disjoint open subintervals of $(0,1)$. Recalling that $\O \subset \hom_K(\Gamma_1,\SO)$, it is obvious that $\rho_t$ is a boundary-compact path on any open subinterval of $I_{\O}$.
On the other hand, we define another open subset of $(0,1)$ by $$J_{\O}=(0,1) - \overline{I_{\O}}.$$
Then $\rho_t$ is a boundary-parabolic path on any open subinterval of $J_{\O}$ because $\rho_t \in \hom_P(\Gamma_1,\SO)$ for all $t\in J_{\O}$.
This leads us to conclude that the derivative of $\vol(\rho_t)$ vanishes on the open subset $I_{\O}\cup J_{\O}$ of $(0,1)$. Furthermore since $I_{\O}\cup J_{\O}$ is an open dense subset of $(0,1)$ and $\vol(\rho_t)$ is $C^1$-smooth on $[0,1]$, the derivative of $\vol(\rho_t)$ vanishes on the whole interval $[0,1]$. Therefore, $\vol(\rho_t)$ is constant.
\end{proof}

\section{General case}\label{sec:general}

In this section, we deal with arbitrary hyperbolic manifolds of finite volume.
Let $\Gamma$ be a nonuniform lattice in $\SO$ and $M=\Gamma \backslash \H$ with $r$ ends.
For each $i=1,\ldots,r$, there is a horoball $U_{\xi_i}$ centered at $\xi_i$ such that an end $E_i$ of $M$ has a neighborhood homeomorphic to $\Gamma_{\xi_i}\backslash U_{\xi_i}$ where $\Gamma_{\xi_i}$ is the stabilizer subgroup of $\xi_i$ in $\Gamma$. By the Margulis lemma, every $\Gamma_{\xi_i}$ is almost nilpotent.

Let $\rho :\Gamma \rightarrow \SO$ be a representation. Then each $\rho(\Gamma_{\xi_i})$ is contained in either a maximal compact subgroup or minimal parabolic subgroup of $\SO$ because $\rho(\Gamma_{\xi_i})$ is amenable. Let $\hom^i_P(\Gamma,\SO)$ be the set of representations which send $\Gamma_{\xi_i}$ into a minimal parabolic subgroup.
Then using a similar argument as in the proof of Lemma \ref{lem:closedset}, it can be shown that $\hom^i_P(\Gamma,\SO)$ is closed in $\hom(\Gamma,\SO)$. Thus we define an open subset $\mathcal O_i = \hom(\Gamma,\SO)- \hom^i_P(\Gamma,\SO)$ which can be presented in the following way.
$$\mathcal O_i = \{ \rho \in \hom(\Gamma,\SO) \ | \ \rho(\Gamma_{\xi_i}) \text{ has a unique fixed point on }\H \}.$$


\begin{definition}
A path $\rho_t : \Gamma \rightarrow \SO$ is said to be a \emph{boundary-compact path with respect to an end $E_i$} if $\rho_t(\Gamma_{\xi_i})$ is contained in a maximal compact subgroup of $\SO$ for each $t$. If $\rho_t(\Gamma_{\xi_i})$ is contained in a minimal parabolic subgroup of $\SO$ for each $t$, we say that $\rho_t$ is a \emph{boundary-parabolic path with respect to an end $E_i$}
\end{definition}

Note that a path on $\hom(\Gamma,\SO)$ may be both a boundary-compact path and boundary-parabolic path with respect to $E_i$ since the intersection of a maximal compact subgroup and a minimal parabolic subgroup may be nonempty.

\begin{lemma}\label{lem:equipath2}Let $\rho_t : \Gamma \rightarrow \SO$ be a $C^1$-smooth path parameterized by $t\in [0,1]$. If $\rho_t$ is either a boundary-compact path or boundary-parabolic path with respect to $E_i$, then there is a $C^1$-smooth path $\eta_i : [0,1] \rightarrow \cH$ such that $\eta_i(t)$ is a fixed point of $\rho_t(\Gamma_{\xi_i})$ for all $t\in [0,1]$. Furthermore, if $\rho_t$ is a boundary-compact path with respect to $E_i$, $\eta_i$ has its image in $\H$ and, if $\rho_t$ is a boundary-parabolic path with respect to $E_i$, $\eta_i$ has its image in $\partial \H$.
\end{lemma}
\begin{proof}  The lemma follows using the same methods as those used in the proof of Lemma \ref{lem:equipath}.
\end{proof}

For each $i=1,\ldots,r$, we define two open subsets $I_{\O_i}$ and $J_{\O_i}$ of $(0,1)$ by
$$I_{\O_i}=\{ t \in (0,1) \ | \ \rho_t \in \O_i \} \text{ and } J_{\O_i}=(0,1)-\overline{I_{\O_i}}.$$
Then it can be easily seen that $V_i=I_{\O_i} \cup J_{\O_i}$ is an open dense subset of $(0,1)$. Obviously $I_{\O_i} \subset \hom_K(\Gamma,\SO)$ and $J_{\O_i}\subset \hom_P(\Gamma,\SO)$. Hence any path defined on a subinterval of $V_i$ is either a boundary-compact path or boundary-parabolic path with respect to $E_i$. Let $V=\cap_{i=1}^r V_i$. Then $V$ is also an open dense subset of $(0,1)$ since $(0,1)$ is a Baire space. Recalling that any open subset of $\mathbb R$ is a countable union of disjoint open intervals, $V$ is a countable union of disjoint open intervals. Let $V=\cup_{k=1}^\infty  I_k$ be the union of disjoint open intervals. A path defined on $I_k$ is either a boundary-compact path or boundary-parabolic path with respect to $E_i$ for all $i=1,\ldots,r$.


\begin{proposition}\label{typepre}
Let $n\geq 4$ and $\rho_t : \Gamma \rightarrow \SO$ be a $C^1$-smooth path.
Assume that for all $i=1,\ldots,r$, the path $\rho_t$ is either a boundary-compact path or  a boundary-parabolic path with respect to $E_i$. Then $\vol(\rho_t)$ is constant.
\end{proposition}
\begin{proof}
We may assume that $\rho_t$ is parameterized by $t\in [0,1]$. It follows from Lemma \ref{lem:equimap} that there is a triangulation $\T$ of $\widehat{\widetilde M}$ and a continuous nondegenerate $\rho_0$-equivariant map $f_0 :\widehat{\widetilde M} \rightarrow \cH$. We use the same notation as in Lemma \ref{lem:equimap}.
By Lemma \ref{lem:equipath2}, there exists a $C^1$-smooth path $\eta_i :[0,1] \rightarrow \cH$ such that $\eta_i(t)$ is a fixed point of $\rho_t(\Gamma_{\xi_i})$ for all $i=1,\ldots,r$. Here note that the image of $\eta_i$ is contained in either $\H$ or $\partial \H$. Then as in the case of a one cusped hyperbolic manifold, we can construct a one-parameter family $f_t :\widehat{\widetilde M} \rightarrow \cH$ satisfying
\begin{itemize}
\item[-] $f_t(v_i)=y_i$ and $f_t(\gamma v_i) = \rho_t(\gamma) y_i$,
\item[-] $f_t(\xi_j)=\eta_j(t)$ and $f_t(\gamma \xi_j)=\rho_t(\gamma) \eta_j(t)$,
\end{itemize}
for all $i=1,\ldots,N$, $j=1,\ldots,r$ and $\gamma \in \Gamma$.

It can be easily seen that $f_t$ is a continuous, nondegenerate, $\rho_t$-equivariant map. Due to the fact that the path $\eta_i$ is a path on either $\H$ or $\partial \H$ for all $i$, the Schl\"{a}fli formula can be applied to a $C^1$-smooth one-parameter family $f_t(s)$ of nondegenerate geodesic simplices like the one-cusped hyperbolic manifolds case for all $s\in \T$. Hence every argument in the proof of Theorem \ref{thm:onecuspedcase} works equally well. Therefore the derivative of $\vol(\rho_t)$ vanishes on $[0,1]$, which implies the proposition.
\end{proof}


Notice that the subpath $\rho_t|_{I_k}$ of $\rho_t$ satisfies the conditions of Proposition \ref{typepre} for all $k \in \mathbb N$. Furthermore, noting that $V=\cup_{k=1}^\infty I_k$ is an open dense subset of $(0,1)$, we obtain the following theorem.

\begin{theorem}\label{pathconstant}
Let $n\geq 4$ and $\Gamma$ be a nonuniform lattice in $\SO$. Let $\rho_t : \Gamma \rightarrow \SO$ be a $C^1$-smooth path on $\hom(\Gamma,\SO)$. Then $\vol(\rho_t)$ is constant.
\end{theorem}
\begin{proof}
The subpath $\rho_t|_{I_k}$ of $\rho_t$ is a $C^1$-smooth path satisfying the conditions of Proposition \ref{typepre} as mentioned above for all $k \in \mathbb N$. According to Proposition \ref{typepre}, the derivative of $\vol(\rho_t)$ vanishes on $I_k$ for all $k$. In other words, on $V$, $$\frac{d \vol(\rho_t)}{dt}=0.$$
Since $\vol(\rho_t)$ is a $C^1$-function on $[0,1]$ and $V$ is an open dense subset of $[0,1]$, the derivative of $\vol(\rho_t)$ vanishes on the whole interval $[0,1]$. We therefore conclude that $\vol(\rho_t)$ is constant.
\end{proof}

Theorem \ref{pathconstant} implies that the volume of a representation takes the same value on a connected component of $\hom(\Gamma,\SO)$. Hence we immediately have the following corollary.

\begin{corollary}\label{compconstant}
Let $n\geq 4$ and $\Gamma$ be a nonuniform lattice in $\SO$. The volume of a representation is constant on each connected component of $\hom(\Gamma,\SO)$. Thus the volume of a representation takes a finite number of values.
\end{corollary}

As an application, we give another proof of the local rigidity theorem by combining Theorem \ref{pathconstant} and the volume rigidity theorem proved by Bucher--Burger--Iozzi in \cite{BBI}. In fact, we prove that a nonuniform lattice in $\SO$ can not be nontrivially deformed in $\hom(\Gamma,\SO)$.

\begin{corollary}\label{cor:globalrigidity}
Let $n\geq 4$ and $\Gamma$ be a nonuniform lattice in $\SO$. Then the connected component of $\mathrm{Hom}(\Gamma,\SO)$ containing a lattice embedding $i :\Gamma \hookrightarrow \SO$ consists of representations conjugate to $i$ by isometries.
\end{corollary}

\begin{proof}
Let $\Gamma$ be a nonuniform lattice in $\SO$ and $i : \Gamma \hookrightarrow \SO$ be a lattice embedding. By Corollary \ref{compconstant}, $\vol(\rho)=\vol(i)$ for any representation $\rho$ in the connected component of $\hom(\Gamma,\SO)$ containing $i : \Gamma \hookrightarrow \SO$. Theorem \ref{BBI} implies that $\rho$ is conjugate to $i$ by an isometry. Therefore the claim follows.
\end{proof}

Note that Corollary \ref{cor:globalrigidity} implies the local rigidity theorem for nonuniform lattices in dimension greater than or equal to $4$.

\section{$3$-dimensional case}\label{sec:3dimension}

In this section, we will deal with the $3$-dimensional case. The main difference from the case of dimension at least $4$ is the absence of Proposition \ref{typepre}. This is originally due to the Schl\"{a}fli formula for ideal simplices in the $3$-dimensional case.

Let $[v_0,v_1,v_2,v_3]$ and $[v_0',v_1',v_2',v_3']$ be geodesic simplices in $\overline{\mathbb H}^3$. We say that the two geodesic simplices are \emph{of the same type} if both $v_i$ and $v_i'$ are either in $\mathbb H^3$ or $\partial \mathbb H^3$ for each $i=0,1,2,3$. Let $S_t=[v^t_0,v_1^t,v_2^t,v^t_3]$ be a smooth one-parameter family of ideal geodesic simplices of the same type in $\overline{\mathbb H}^3$. Choose small horoballs centered at the ideal vertices of $S_0$ to be disjoint and then, truncate $S_0$ by cutting off the horoballs. Note that any edge $e$ with infinite length becomes an edge with finite length by cutting off the horoballs. Let $l(e)$ be the length of the truncated edge. Then the first order variation of the volume of $S_t$ at $t=0$ is \begin{eqnarray}\label{eqn:3formula} \frac{d \vol (S_t)}{dt} \bigg|_{t=0} = -\frac{1}{2} \sum_{e} l(e) \frac{d \theta_e(t)}{dt}\bigg|_{t=0},\end{eqnarray}
where the sum is over all edges of $S_0$ and $\theta_e$ denotes the dihedral angle at $e$.

To deal with the representation variety of a more general $3$-manifold in $\PSL$, we recall the definition of cusped manifolds.
\begin{definition}
An orientable manifold $M$ is called a \emph{cusped manifold} if it is diffeomorphic to the interior of a compact manifold with boundary $\oM$. A cusp of $M$ is a closed regular neighborhood of a component of $\partial \oM$. In the following, $M$ is required to have dimension $3$ and $\partial \oM$ to be a union of tori, so each cusp is homeomorphic to $T \times [0,\infty)$ where $T$ denotes a torus.
\end{definition}

Each cusp of $M$ defines a subgroup of $\pi_1(M)$, which is well-defined up to conjugacy. These are called the \emph{peripheral subgroups} of $M$. Since $\pi_1(\partial \oM)$ is abelian, one can define the volume $\vol(\rho)$ of a representation $\rho :\pi_1(M) \rightarrow \PSL$ as usual. Noting that every $3$-manifold is triangulable, every argument in Section \ref{sec:volume} works in this setting. Hence from now on we stick to the notations used in Section \ref{sec:volume}.

In contrast to the case of dimension greater than or equal to $4$, the volume of a representation in the $3$-dimensional case takes an infinite number of values. In fact, the set of values contains an open interval: let $M=\Gamma\backslash \mathbb H^3$ be a noncompact hyperbolic $3$-manifold of finite volume and let $\hom_0(\pi_1(M),\PSL)$ be the connected component of $\hom(\pi_1(M),\PSL)$ containing the holonomy of a hyperbolic structure on $M$. Then there is a representation $\rho'$ in $\hom_0(\pi_1(M),\PSL)$ which factors through the fundamental group of a Dehn filling of $M$ such that $$\vol(i)=\vol(M) > \vol(\rho').$$
Since the volume of a representation is continuous on $\hom(\pi_1(M),\PSL)$, the set of values for the volume of a representation contains an open interval $(\vol(\rho'),\vol(i))$.

However, there is a subvariety of $\hom(\pi_1(M),\PSL)$ such that the volume of a representation is constant on each connected component of the subvariety. See Corollary \ref{compconstant3}. This subvariety is closely related to the peripheral subgroups of $M$.

\begin{definition} A representation $\rho : \pi_1(M) \rightarrow \PSL$ is called \emph{boundary-parabolic} if $\rho$ sends each peripheral subgroup of $M$ to a parabolic subgroup. \end{definition}

Let $\hom_{par}(\pi_1(M),\PSL)$ be the set of boundary-parabolic representations. It is actually an algebraic subvariety of $\hom(\pi_1(M),\PSL)$. Thus $\hom_{par}(\pi_1(M),\PSL)$ has finitely many connected components. The following theorem is analogous to Theorem \ref{pathconstant} for $\hom_{par}(\pi_1(M),\PSL)$:

\begin{theorem}\label{pathconstant3}
Let $\Gamma$ be a nonuniform lattice in $\mathrm{SO}(3,1)$ and $\rho_t : \Gamma \rightarrow \mathrm{SO}(3,1)$ a $C^1$-smooth path on $\mathrm{Hom}_{par}(\Gamma, \mathrm{SO}(3,1))$. Then $\vol(\rho_t)$ is constant.
\end{theorem}
\begin{proof}
Let $\partial_1 \overline M,\ldots,\partial_r \overline M$ be the connected components of $\overline M$ and $L_1,\ldots,L_r$ be the associated peripheral subgroups of $\pi_1(M)$. Since every $\rho_t(L_i)$ is a parabolic subgroup, $\rho_t(L_i)$ has a unique fixed point on $\partial \mathbb H^3$. Hence, to each peripheral subgroup $L_i$, we can associate a $C^1$-smooth path $\eta_i :[0,1]\rightarrow \partial \mathbb H^3$ such that $\eta_i(t)$ is the unique fixed point of $\rho_t(L_i)$.

Since Lemma \ref{lem:equimap} works for cusped manifolds as well, there is a triangulation $\T$ and a continuous nondegenerate $\rho_0$-equivariant map $f_0: \widehat{\widetilde M} \rightarrow \overline{\mathbb H}^3$. We stick to the notation used in the proof of Lemma \ref{lem:equimap}. Each peripheral subgroup $L_i$ is the stabilizer subgroup of an ideal point of $\widehat{\widetilde M}$ in $\pi_1(M)$. Denote by $\xi_i$ the fixed point of $L_i$ on $\widehat{\widetilde M}$. Then we define a one-parameter family $f_t :\widehat{\widetilde M} \rightarrow \overline{\mathbb H}^3$ satisfying
\begin{itemize}
\item[-] $f_t(v_i)=y_i$ and $f_t(\gamma v_i) = \rho_t(\gamma) y_i$,
\item[-] $f_t(\xi_j)=\eta_j(t)$ and $f_t(\gamma \xi_j)=\rho_t(\gamma) \eta_j(t)$,
\end{itemize}
for all $i=1,\ldots,N$, $j=1,\ldots,r$ and $\gamma \in \pi_1(M)$.
Then $f_t$ is a continuous, nondegenerate, $\rho_t$-equivariant map and moreover, it sends an ideal point of $\widehat{\widetilde M}$ to an ideal point of $\overline{\mathbb H}^3$.

Let $\bar e$ denote an edge of $\T_{\widehat M}$ and $\bar s$ a simplex of $\T_{\widehat M}$.
Then each simplex $\bar s$ gives rise to a smooth one-parameter family $f_t(s)$ of geodesic simplices of $\overline{\mathbb H}^3$ of the same type, where $s$ is a lift of $\bar s$ to $\widehat{\widetilde M}$.
Hence it is possible to apply the Schl\"{a}fli formula to the family $f_t(s)$ for any $\bar s \in \T_{\wM}$.
Following the same computation as in Theorem \ref{pathconstant}, we have
\begin{eqnarray}\label{eqn:dervolume}
\frac{d \vol(\rho_t)}{dt}   \bigg|_{t=0} = -\frac{1}{2} \sum_{\bar e} \left( \sum_{\bar s / \bar e \subset \bar s} \epsilon(\bar s)  \frac{d \theta (t;\bar e,\bar s)}{dt} \bigg|_{t=0} \right) l(\bar e),
\end{eqnarray}
where $l(\bar e)$ is the length of the edge obtained by truncating $f_0(s)$ in the way described in the beginning of this section.

Every edge $\bar e$ with finite length makes a trivial contribution to the computation of the differential of the volume of a representation, i.e., $$ \sum_{\bar s / \bar e \subset \bar s} \epsilon(\bar s)  \frac{d \theta (t;\bar e,\bar s)}{dt} \bigg|_{t=0} =0.$$
Since it is possible to choose $\rho_t(\pi_1(M))$-invariant horospheres at the ideal
points $\cup_{i=1}^r \rho_t(\pi_1(M))\eta_i(t)$ when truncating the simplices of the images of $\T$ by $f_t$ for all $t$, it is possible to define the degree of $f_t$ to an edge with infinite length.
Furthermore, it can be seen, as in case where the dimension is greater than or equal to $4$, that the degree of $f_t$ to an edge with infinite length does not change. Hence every edge with infinite length also makes no contribution in Equation (\ref{eqn:dervolume}). Note that in general, it is not possible to choose such invariant horospheres. However, in our situation, it is possible due to the condition that $\rho_t$ map each peripheral subgroup to a parabolic subgroup. Finally we have $$\frac{d \vol(\rho_t)}{dt} =0,$$
which completes the proof.
\end{proof}

The following corollary immediately follows from Theorem \ref{pathconstant3}.

\begin{corollary}\label{compconstant3}
The volume of a representation is constant on each connected component of $\hom_{par}(\pi_1(M),\PSL)$.
\end{corollary}

Let $\Gamma$ be a nonuniform lattice in $\PSL$.
Then a lattice embedding $i :\Gamma \hookrightarrow \PSL$ is not locally rigid in $\hom(\Gamma,\PSL)$. However, there is no local deformation of the lattice embedding $i$ in $\hom_{par}(\Gamma,\PSL)$. This was proved by verifying $H^1_{par}(\Gamma, Ad\circ i)=0$. Our approach makes it possible to give a proof for this local rigidity theorem in $\hom_{par}(\Gamma,\PSL)$ as follows.

\begin{corollary}
Let $\Gamma$ be a nonuniform lattice in $\mathrm{SO}(3,1)$. Then the connected component of $\mathrm{Hom}_{par}(\Gamma,\mathrm{SO}(3,1))$ containing a lattice embedding $i :\Gamma \hookrightarrow \mathrm{SO}(3,1)$ consists of representations conjugate to $i$ by isometries.
\end{corollary}
\begin{proof}
Let $\Gamma$ be a nonuniform lattice in $\PSL$ and $M=\Gamma \backslash \mathbb H^3$. Let $\hom_{par}^0(\Gamma,\PSL)$ be the connected component of $\hom_{par}(\Gamma,\PSL)$ containing a lattice embedding $i:\Gamma \hookrightarrow \PSL$. Corollary \ref{compconstant3} implies that for any representation $\rho \in \hom_{par}^0(\Gamma,\PSL)$, $$\vol(\rho)=\vol(i)=\vol(M).$$
By Theorem \ref{BBI}, $\rho$ is conjugate to $i$ by an isometry.
\end{proof}

As another application, we obtain, for noncompact $3$-manifolds, an analogue of Soma's theorem.

\begin{theorem}
Let $N$ be a connected, orientable, cusped $3$-manifold. Then there are only a finite number of hyperbolic $3$-manifolds dominated by $N$.
\end{theorem}
\begin{proof}
Suppose that a hyperbolic manifold $M$ is dominated by $N$. Then there exists a proper map $f: N \rightarrow M$ with nonzero degree. For any orientable manifold, its fundamental class is well-defined in its locally finite homology. We refer the reader to \cite[Section 2]{BKK} for more details about locally finite homology. Let $[N]$ and $[M]$ be the locally finite fundamental classes of $N$ and $M$ respectively.
Then the degree of $f$, denoted by $\mathrm{deg}f$, is defined as the integer satisfying $$f_*[N]=(\mathrm{deg} f)[M].$$ Then we have
\begin{eqnarray}\label{eqn:degree} |\mathrm{deg}f | \cdot  \|M \| = \| f_* [N] \| \leq \| N \| \end{eqnarray}
where $\|N \|$ denotes the simplicial volume of $N$.

Since the fundamental group of $\partial N$ is amenable, $\| N \|$ is finite.
This follows from the finiteness criterion for simplicial volume in \cite[Theorem 6.1]{Loh}.
Hence $\| M \|$ is finite and so is the volume of $M$. Moreover the proportionality principle in \cite{Th78} for the simplicial volume of hyperbolic $3$-manifolds implies that $$\|M\|=\frac{\vol(M)}{v_3},$$ where $v_3$ is the volume of a regular ideal geodesic simplex in $\mathbb H^3$.
It is a standard fact that there is a lower bound to the volume of hyperbolic $3$-manifolds.
From this fact and Equation (\ref{eqn:degree}), it follows that $| \mathrm{deg}f |$ is uniformly bounded from above. In other words, there are only finitely many possible values for $\mathrm{deg}f$.

A proper map $f:N\rightarrow M$ with nonzero degree gives rise to a representation $\rho : \pi_1(N) \rightarrow \PSL$ which sends each peripheral subgroup of $\pi_1(N)$ to a parabolic subgroup. Thus $\rho \in \hom_{par}(\pi_1(N),\PSL)$. It can be easily seen that $$\vol(\rho)= (\mathrm{deg}f ) \cdot \vol(M).$$
As was seen in Corollary \ref{compconstant3}, there are only finitely many possible values for $\vol(\rho)$. This leads us to conclude that there are only finitely many possible values for $\vol(M)$. Since the volume is a finite-to-one function of hyperbolic manifolds, we finally conclude that there are only finitely many possible hyperbolic $3$-manifolds for $M$.
\end{proof}

\section{$2$-dimensional case}\label{sec:2dimension}

The representation variety of closed surface groups in $\mathrm{SO}(2,1)$ has been intensively studied. Recently, Burger--Iozzi--Wienhard \cite{BIW10} developed the theory of a noncompact surface representation variety in the language of bounded cohomology. In this section, we explore a noncompact surface representation variety in our setting.

Let $\Sigma$ denote a surface of genus $g$ with $r$ punctures. We assume that $\Sigma$ has a negative Euler number. Then $\pi_1(\Sigma)$ has a representation $$\pi_1(\Sigma)=\left\langle a_1,b_1,\ldots,a_g,b_g,c_1,\ldots,c_r \ \Bigg| \ \prod_{i=1}^g [a_i, b_i] \prod_{j=1}^r c_j = id \right\rangle. $$
Since $\pi_1(\Sigma)$ is a free group with $(2g+r-1)$ generators, the representation variety $\hom(\pi_1(\Sigma),\mathrm{SO}(2,1))$ is diffeomorphic to $\mathrm{SO}(2,1)^{2g+r-1}$. The volume of a representation $\rho :\pi_1(\Sigma) \rightarrow \mathrm{SO}(2,1)$ is referred to as \emph{the Toledo number} of $\rho$. The range of the volume of a representation is actually the closed interval $[-2\pi (2g-2+r)  , 2\pi (2g-2+r) ]$.

As we described in Section \ref{sec:volume}, the volume of a representation $\rho :\pi_1(\Sigma) \rightarrow \mathrm{SO}(2,1)$ can be computed by $\vol(\rho)=\langle f^* \bar \omega_2, [\widehat \Sigma] \rangle$,
where $f: \widehat{\widetilde \Sigma} \rightarrow \overline{\mathbb H}^2$ is a continuous $\rho$-equivariant map. Take a triangulation $\T$ and the continuous, nondegenerate, $\rho$-equivariant map $f: \widehat{\widetilde \Sigma} \rightarrow \overline{\mathbb H}^2$ described in Lemma \ref{lem:equimap}. Let $\widehat z =\sum_{\bar s \in \T_{\widehat \Sigma}} \bar s$ be the fundamental cycle induced from the triangulation $\T_{\widehat \Sigma}$ on $\widehat \Sigma$. Then
\begin{eqnarray}\label{eqn:areaeqn} \vol(\rho)= \sum_{\bar s \in \T_{\widehat \Sigma}} \vol_2 \left( f(s(e_0)),f(s(e_1)),f(s(e_2)) \right), \end{eqnarray}
where $s$ is a lift of $\bar s$ to $\widehat{\widetilde \Sigma}$.

Recall that $f(s)$ has at most one ideal vertex for any $\bar s \in \T_{\widehat \Sigma}$. The area of a geodesic $2$-simplex $\sigma$ with internal angles $\theta_1, \theta_2, \theta_3$ is $\pi-(\theta_1+\theta_2+\theta_3)$. Let $\mathrm{A}(\sigma)=\theta_1+\theta_2+\theta_3$. Then we reformulate Equation (\ref{eqn:areaeqn}) as follows:
\begin{eqnarray*}
\vol(\rho) &=& \sum_{\bar s} \epsilon(\bar s)(\pi - \mathrm{A}(\bar s)) \\
&=& \sum_{\bar s} \epsilon(\bar s) \pi - \sum_{\bar v}  \sum_{\bar s/\bar v \subset \bar s} \epsilon(\bar s) \theta(\bar v, \bar s ),
\end{eqnarray*}
where $\bar v$ denotes a vertex of $\T_{\widehat \Sigma}$ and $\theta(\bar v, \bar s )$ is the dihedral angle of $\bar s$ at $\bar v$.

Let $v$ be a lift of $\bar v \in \T_{\widehat \Sigma}$ to $\widehat{\widetilde \Sigma}$. Then $f$ induces a homomorphism $$(f_v)_* : H_1(N_v-\{v\},\mathbb Z) \rightarrow H_1(N_{v'}-\{v'\},\mathbb Z)$$
where $N_v$ denotes a ball neighborhood of $v$ and $v'=f(v)$. Since $$H_1(N_v-\{v\},\mathbb Z) \cong H_1(N_{v'}-\{v'\},\mathbb Z) \cong \mathbb Z,$$
one can define the degree of $f$ at $v$ as usual, denoted by $\textrm{deg}_vf$. Then the following lemma is obvious.

\begin{lemma}\label{lem:2deg}
For any interior vertex $v \in \mathbb H^2$, $$2\pi \mathrm{deg}_v f= \pm \sum_{\bar s/\bar v \subset \bar s} \epsilon(\bar s) \theta(\bar v, \bar s ).$$
\end{lemma}

Lemma \ref{lem:2deg} implies that every interior vertex of $\T_{\widehat \Sigma}$ makes a trivial contribution to the computation of the differential of the volume of a representation. In other words, the ideal vertices of $\T_{\widehat \Sigma}$ determine the differential of the volume of  representations. However, an ideal vertex of $\T_{\widehat \Sigma}$ with its image in $\partial \mathbb H^2$ does not also contribute to the computation of the differential of the volume of a representation since the angle of every ideal geodesic simplex in $\mathbb H^2$ at any ideal vertex is zero. Hence consider a subset of $\hom(\pi_1(\Sigma),\mathrm{SO}(2,1))$ defined by
\begin{eqnarray*} \lefteqn{\hom_\partial(\pi_1(\Sigma),\mathrm{SO}(2,1))=\{ \rho  \ | \ \rho(c_i) \text{ has at least}} \\ & & \ \ \ \ \ \ \ \ \ \ \ \ \text{one fixed point in }\partial \mathbb H^2 \text{ for all }i=1,\ldots,r \}.\end{eqnarray*}
The set $\hom_\partial(\pi_1(\Sigma),\mathrm{SO}(2,1))$ is a semialgebraic set. Hence it has finitely many connected components. In fact, Burger--Iozzi--Wienhard \cite{BIW10} proved that the volume of a representation is constant on each connected component of $\hom_\partial(\pi_1(\Sigma),\mathrm{SO}(2,1))$. We give a simple proof for this from the point of our view.

\begin{theorem}\label{thm:2hyp}
The volume of a representation is constant on each connected component of $\hom_\partial(\pi_1(\Sigma),\mathrm{SO}(2,1))$.
\end{theorem}

\begin{proof}
Let $\rho_t :\pi_1(\Sigma) \rightarrow \mathrm{SO}(2,1)$ be a $C^1$-smooth path.
There is a one-parameter family $f_t : \widehat{\widetilde \Sigma} \rightarrow \overline{\mathbb H}^2$ such that $f_t$ is a continuous, nondegenerate $\rho_t$-equivariant map and $f_t$ sends  ideal points on $\widehat{\widetilde \Sigma}$ to ideal points on $\partial \mathbb H^2$. Then $$\vol(\rho_t)= \sum_{\bar s} \epsilon(\bar s) \pi - \sum_{\bar v}  \sum_{\bar s/\bar v \subset \bar s} \epsilon(\bar s) \theta(t;\bar v, \bar s )$$
where $\theta(t;\bar v, \bar s )$ denotes the dihedral angle of $f_t(s)$ at $f_t(v)$.

If $\bar v$ is an ideal vertex, then $\theta(t;\bar v, \bar s )=0$ for all $\bar s$ containing $\bar v$  because $f_t(v)$ is on $\partial \mathbb H^2$, i.e., $$\sum_{\bar s/\bar v \subset \bar s} \epsilon(\bar s)  \theta(t;\bar v, \bar s )=0.$$
If $\bar v $ is not an ideal vertex, $f_t$ does not change either the degree at $\bar v$ or the orientation on $\bar s$ for any $\bar s \in \T_{\widehat \Sigma}$ since $f_t(s)$ is a $C^1$-smooth one-parameter family of geodesic simplices. In other words, $$\sum_{\bar s/\bar v \subset \bar s} \epsilon(\bar s)  \theta(t;\bar v, \bar s )=\sum_{\bar s/\bar v \subset \bar s} \epsilon(\bar s) \theta(t';\bar v, \bar s ),$$ for any $t, t' \in [0,1]$.
This leads us to conclude that $\vol(\rho_t)$ is constant, which implies the theorem.
\end{proof}

\end{document}